\documentclass[12pt]{article}
\usepackage{amsmath,graphicx,amssymb,amsthm}
\usepackage{rotating}

\allowdisplaybreaks

\def\amslatex{$\mathcal{A}\kern-.1667em\lower.5ex\hbox{$\mathcal{M}$}\kern-.125em\mathcal{S}$-\LaTeX}

\newtheorem{set}{set}[section]
\newtheorem{Corollary}[set]{Corollary}

\newtheorem{Lemma}[set]{Lemma}

\newtheorem{Proposition}[set]{Proposition}
\newtheorem{Remark}[set]{Remark}
\newtheorem{Theorem}[set]{Theorem}
\newcommand{\define}{\mathrel{\hbox{$\equiv$\hskip -.90em \lower .47ex \hbox{$\leftharpoondown$}}}}
\newcommand{\enifed}{\mathrel{\hbox{$\equiv$\hskip -.90em \lower .47ex \hbox{$\rightharpoondown$}}}}

\numberwithin{equation}{section}
\begin{document}
\title{On a class of operators in the hyperfinite ${\rm II}_1$ factor}
\author{Zhangsheng Zhu \and Junsheng Fang  \and Rui Shi}
\date{}
\maketitle

 \begin{abstract}
Let $R$ be the hyperfinite  ${\rm II}_1$ factor and let $u,v$
be two generators of $R$ such that $u^*u=v^*v=1$ and $vu=e^{2\pi
i\theta} uv$ for an irrational number $\theta$. In this paper we study the class
of operators $uf(v)$, where $f$ is a bounded Lebesgue measurable function on the unit
circle $S^1$. We
calculate the spectrum and Brown spectrum of operators $uf(v)$, and study the invariant
subspace problem of such operators relative to $R$. We show that under general assumptions the von Neumann algebra generated
by $uf(v)$ is an irreducible subfactor of $R$ with index $n$ for some natural number $n$, and the $C^*$-algebra
generated by $uf(v)$ and the identity operator is a generalized universal irrational rotation $C^*$-algebra.
 \end{abstract}

\section{Introduction}

Let $M$ be a von Neumann algebra acting on a Hilbert space $\mathcal{H}$. A closed subspace $\mathcal{K}$ of $\mathcal{H}$
is said to be affiliated with $M$ if the projection of $\mathcal{H}$ onto $\mathcal{K}$ belongs to $M$. For $T\in M$, a subspace
$K$ is said to be $T$-invariant if $T\mathcal{K}\subseteq \mathcal{K}$ or equivalently $P_{\mathcal{K}}TP_{\mathcal{K}}=TP_{\mathcal{K}}$.
The invariant subspace problem relative to a von Neumann algebra $M$ asks whether every operator $T\in M$ has a non-trivial, closed, invariant
subspace $\mathcal{K}$ affiliated with $M$, and the hyperinvariant subspace problem asks whether one can always choose such a $\mathcal{K}$ to be
hyperinvariant for $T$, i.e., it is $S$-invariant for every $S\in \mathcal{B}(\mathcal{H})$ that commutes with $T$. If the subspace $\mathcal{K}$ is $T$-hyperinvariant, then $P_{\mathcal{K}}\in W^*(T)=\{T,T^*\}''$.

Let $M$ be a finite von Neumann algebra with a faithful normal
tracial state $\tau$. The \emph{Fuglede-Kadison
determinant}~\cite{FK}, $\Delta:\, M\rightarrow [0,+\infty[,$ is
given by
\[
\Delta(T)={\rm exp}\{\tau(\ln|T|)\},\quad T\in M,
\]
with ${\rm exp}\{-\infty\}:=0$. For an arbitrary element $T$ in
$M$ the function $\lambda\rightarrow \ln\Delta(T-\lambda 1)$ is
subharmonic on $\mathbb{C}$, and its Laplacian
\[
d\mu_T(\lambda):=\frac{1}{2\pi}\triangledown^2\ln \Delta(T-\lambda
1),
\]
in the distribution sense, defines a probability measure $\mu_T$
on $\mathbb{C}$, called the \emph{Brown's spectral distribution}
or \emph{Brown measure} of $T$~\cite{Br}. From the definition, Brown
measure $\mu_T$ only depends on the joint distribution of $T$ and
$T^*$, i.e., the (noncommutative) mixed moments of $T$ and $T^*$.
If $T$ is normal, then $\mu_T$ is the  trace $\tau$ composed with
the spectral projections of $T$. If $M=M_n(\mathbb{C})$, then
$\mu_T$ is the normalized counting measure
$\frac{1}{n}\left(\delta_{\lambda_1}+\delta_{\lambda_2}+\cdots+\delta_{\lambda_n}\right)$,
where $\lambda_1,\lambda_2,\cdots,\lambda_n$ are the eigenvalues
of $T$ repeated according to root multiplicity.
Recently, Uffe Haagerup and Hanne Schultz made a huge advance on the invariant subspace problem relative to a type ${\rm II}_1$ factor(\cite{HS2}). They proved that if the Brown measure of an operator $T$ in a type ${\rm II}_1$ factor is not concentrated in one point, the operator $T$ has a non-trivial, closed, invariant subspace $\mathcal{K}$ affiliated with $M$ and moreover, this subspace is hyperinvariant. However,
the calculation of Brown measures of nonnormal operators is complicated in general (see\cite{HL,BL,FHM}). Note that the support of the Brown measure of an operator is contained in the spectrum of the operator.

As regards the invariant subspace problem relative to the von Neumann algebra, the following question remains open: If $T$ is an operator in a type ${\rm II}_1$ factor $M$ and if the Brown measure $\mu_T$ is a Dirac measure, for example if $T$ is quasinipotent, does $T$ has a non-trivial, closed,
invariant subspace affiliated with $M$? In~\cite{DH}, Dykema and Haagerup introduced the family of ${\rm DT}$-operators and they studied many of their properties. In~\cite{DH2} they showed that every quasinilpotent ${\rm DT}$-operator $T$ has a one-parameter family of non-trivial hyperinvariant subspaces. In particular, they proved that for $t\in [0,1]$
\[
\mathcal{H}_t=\left\{\xi\in \mathcal{H}:\, \limsup_n\left(\frac{k}{e}\left\|T^k\xi\right\| \right)^{2/k}\leq t \right\}
\]
is a closed, hyperinvariant subspace of $T$. In~\cite{Ga}, Gabriel introduced a class of quasinilpotent operators in the hyperfinite ${\rm II}_1$ factor $R$.  He showed that the quasinilpotent operator generates $R$ and it has non-trivial, closed invariant subspaces affiliated to $R$. However the existence of nontrivial hyperinvariant subspaces of such class of operators remains open.

Let $R$ be the hyperfinite ${\rm II}_1$ factor and let $\theta\in (0,1)$ be an irrational
number. Then there are two unitary operators $u,v$ in $R$ such that $R=\{u,v\}''$ and $vu=e^{2\pi
i\theta} uv$. In this paper we study the class
of operators $uf(v)$ in $R$, where $f$ is a bounded Lesbegue measurable function on the unit
circle $S^1$.  A natural example of the class of operators is $u+\lambda v$, where $\lambda\in \mathbb{C}$. Indeed, if let $w=u^*v$, then
$R=\{u,w\}''$ and $wu=e^{2\pi i\theta} uw$. Note that $u+\lambda v=u(1+\lambda w)$. The
operator $u+\lambda v$ is closely related to the so called almost
Mathiew operators which can be viewed as the operator $(u+\lambda
e^{2\pi i\beta} v)+(u+\lambda e^{2\pi i\beta} v)^*$ in $R$ (see
\cite{La} for a recent historical account and for the physics
background of almost Mathiew operators).

The above class of operators are analogues of $R$-diagonal operators. Recall that if $u$ and $v$
are free Haar unitary operators in a finite von Neumann algebra $M$ and $f$ is a bounded measurable function on the unit
circle $S^1$ then $uf(v)$ is an $R$-diagonal operator(\cite{NS}). In~\cite{HL}, Haagerup and Larson calculated
the spectrum and Brown spectrum of $R$-diagonal operators.
 In~\cite{SS}, Sniady and Speicher proved that
every $R$-diagonal operator has a continuous family of invariant subspaces affilated with $M$.

In  section 2 and section 3 of this paper, we calculate the spectrum of $uf(v)$ in $R$, where $f$ is a continuous function on $S^1$.  The main result is that the spectrum of $uf(v)$ is given by
\[
\sigma(uf(v))=
\begin{cases}
\Delta(f(v))S^1& f(v)\,\text{is invertible},\\
\overline{{\mathbb B}(0,\Delta(f(v)))}&f(v)\,\text{is not
invertible},
\end{cases}
\]
where
$\Delta(f(v))=\exp(\int_0^1 \ln|f(e^{2\pi ix})|dx)$ is the
Fuglede-Kadison determinant of $f(v)$.  In section 2 we show that
the spectral radius of $uf(v)$ is $\Delta(f(v))$. A key idea in
the calculation  is using Birkhoff's Ergodic theorem and the unique ergodicity of the irrational rotation. Then in
section 3 we prove the main result. The main difficulty is to show
that $\sigma(uf(v))$ is connected. This is done by using the
averaging technique. We also point out the main result does not hold for some $f\in L^\infty(S^1,m)$.

In section 4, we study the von Neumann algebra generated by $uf(v)$. We show that if the zero set of $f(z)\in L^\infty(S^1,m)$ has Lebesgue measure zero, then $W^*(uf(v))$ is an irreducible subfactor of $R$ with index $n$ for some positive integer $n$.

In section 5, we consider the invariant subspace problem of $uf(v)$ relative to $R$. Firstly we calculate the Brown measure of $uf(v)$. We will
show that the Brown measure of $uf(v)$ (in $R$) is  the Haar measure on $\Delta(f(v))S^1$ for all $f\in L^\infty(S^1,m)$. As a corollary of Haagerup and Schultz's result, if $\Delta(f(v))>0$, for example $f$ is a polynomial, then $uf(v)$ has a continuous family of invariant subspaces affiliated with $M$. On the other hand, if $\Delta(f(v))=0$ we show that the known methods are unable to determine whether or not the operator $uf(v)$ has a nontrivial, closed, invariant subspace affiliated with $R$. Thus such class of operators are interesting candidates for the question of the invariant subspace problem relative to $R$.

 Recall that a generalized universal irrational rotation $C^*$-algebra  $A_{\theta,\gamma}$ is the universal $C^*$-algebra generated by $x$ and $w$ satisfying the following properties~(\cite{FJX}):
\begin{equation}\label{U1}
w^*w=ww^*=1,
\end{equation}
\begin{equation}\label{U2}
x^*x=\gamma(w),
\end{equation}
\begin{equation}\label{U3}
xx^*=\gamma(e^{-2\pi i\theta}w),
\end{equation}
\begin{equation}\label{U4}
xw=e^{-2\pi i\theta}wx,
\end{equation}
where $\gamma(z)\in C(S^1)$ is a positive function. If $\gamma(z)\equiv 1$ (or $\gamma(z)$ does not have zero points), then $A_{\theta,\gamma}$ is the universal irrational rotation $C^*$-algebra. In~\cite{FJX}, many properties of  generalized universal irrational rotation $C^*$-algebras are studied, including tracial state spaces, simplicity, $K$-groups, and classification of simple generalized universal irrational rotation $C^*$-algebras. For instance, the following results are Theorem 5.7 and Theorem 6.6 of~\cite{FJX} respectively.

\begin{Theorem}\label{T:K-groups}
Let $Y$ be the set of zeros of $\gamma.$ If $\emptyset\neq Y\neq S^1$, then
\[
K_1({A}_{\theta, \gamma})\cong{\mathbb Z}
\]
and there exists a splitting short exact sequence:
\[
0\to {\mathbb Z}\to K_0({A}_{\theta, \gamma})\to C(Y, {\mathbb Z})\to 0.
\]
In particular, if $Y$ has $n$ points, then
\[
K_0({A}_{\theta, \gamma})\cong {\mathbb Z}^{n+1}.
\]
\end{Theorem}

\begin{Theorem}\label{T1c1}
Let $\theta_1$ and $\theta_2$ be two irrational numbers,
$\gamma_1$ and $\gamma_2\in C(S^1)$ be non-negative functions and let $Y_i$ be the set of zeros of $\gamma_i,$ $i=1,2.$
Suppose that $A_{\theta_i, \gamma_i}$ is simple.
Then  $A_{\theta_1, \gamma_1}\cong A_{\theta_2, \gamma_2}$  if and only if
the following hold:
\[
\theta_1=\pm \theta_2 mod(\mathbb Z)\quad{\rm and}\quad C(Y_1, {\mathbb Z})/{\mathbb Z}\cong C(Y_2, {\mathbb Z})/{\mathbb Z}.
\]
In particular, when  $\gamma_1$ has only finitely many zeros, then
$A_{\theta_1,\gamma_1}\cong A_{\theta_2, \gamma_2}$  if and only if $\theta_1=\pm \theta_2 mod{\mathbb{Z}}$ and
$\gamma_2$ has the same number of zeros.
\end{Theorem}
In section 6, we show that the $C^*$-algebra generated by $uf(v)$ and the identity operator is closely related to the  generalized universal irrational rotation algebra. Precisely, we will prove the following result. Let $Y$ be the zero points of $f(z)$. If $Y$ satisfy $\phi^n(Y)\cap Y=\emptyset$ for any integer $n\not=0$, where $\phi(z)=e^{2\pi i\theta}z$, then $C^*(uf(v),1)$ is a generalized universal irrational rotation $C^*$-algebra. Furthermore, if $|f|(z)$ is not a periodic function, then $C^*(uf(v),1)\cong A_{\theta,|f|^2}$. As a corollary, we will show that if $A_{\theta,\gamma}$ is a simple $C^*$-algebra, then $A_{\theta,\gamma}$ is generated by an element $uf(v)$ and the identity operator for some $f(z)\in C(S^1)$.

\noindent{\bf Acknowledgement:}\, The authors thank Chunlan Jiang and Liu zhengwei for valuable discussions on the paper.

\section{Spectral radius of $uf(v)$}
Let $\alpha=e^{2\pi i\theta}$. Since $vu=\alpha uv$,
 $f(v)u=uf(\alpha v)$ for all $f\in
L^\infty(S^1,m)$. So
\[
(uf(v))^2=uf(v)uf(v)=u^2f(\alpha v)f(v),
\]
\[
(uf(v))^3=uf(v)uf(v)uf(v)=u^3f(\alpha^2v)f(\alpha v)f(v).
\]
By induction, we have
\[
(uf(v))^n=u^nf(\alpha^{n-1}v)f(\alpha^{n-2}v)\cdots f(v).
\]
Let $r(uf(v))$ be the spectral radius of $uf(v)$. Then
\[
r(uf(v))=\lim_{n\rightarrow+\infty} \|(uf(v))^n\|^{1/n}=\lim_{n\rightarrow+\infty} \|f(\alpha^{n-1}v)f(\alpha^{n-2}v)\cdots f(v)\|^{1/n}.
\]
Since $v$ is a Haar unitary operator, we may identify $v$ with the multiplication operator $M_z$ on $L^2(S^1,m)$, where $m$ is the Haar measure on $S^1$. Hence,
\begin{eqnarray}
\nonumber \|(uf(v))^n\|^{1/n}&=&\|f(\alpha^{n-1}v)f(\alpha^{n-2}v)\cdots f(v)\|^{1/n}\\
\nonumber &=&\|f(\alpha^{n-1}z)f(\alpha^{n-2}z)\cdots f(z)\|_\infty^{1/n}\\
\nonumber&=&\left({\rm esssup}_{z\in
S^1}\left|f(\alpha^{n-1}z)f(\alpha^{n-2}z)\cdots
f(z)\right|\right)^{1/n}.
\end{eqnarray}

\begin{Lemma}\label{L:leq}
If $f(z)\in L^\infty(S^1,m)$ and $\int_0^1\left|\left(\ln
|f(e^{2\pi ix})|\right)\right|dx<\infty$, then $r(uf(v))\geq
\Delta(f(v))$.
\end{Lemma}
\begin{proof}
Let $T:x\rightarrow x+\theta (\text{mod}\, 1)$. Then $T$ is a measure preserving ergodic
transformation of $[0,1]$. By  Birkhoff's Ergodic theorem and the assumption
 $\int_0^1|(\ln |f(e^{2\pi ix})|)|dx<\infty$, for almost all $x\in [0,1]$
\[
\lim_{n\rightarrow\infty}\frac{1}{n}\sum_{k=0}^{n-1}\ln|f(\alpha^k
e^{2\pi ix})|=\int_0^1\ln |f(e^{2\pi ix})|dx.
\]
Let $\epsilon>0$. Then there is a measurable subset $E$ of $[0,1]$
with $m(E)>0$ and $N\in \mathbb{N}$ such that for all $x\in E$ and
$n\geq N$
\[
\frac{1}{n}\sum_{k=0}^{n-1}\ln|f(\alpha^k e^{2\pi ix})|\geq
\int_0^1\ln |f(e^{2\pi ix})|dx-\epsilon,
\]
i.e.,
\[
\left|f(\alpha^{n-1}e^{2\pi ix})f(\alpha^{n-2}e^{2\pi ix})\cdots
f(e^{2\pi ix})\right|^{1/n}\geq \exp\left(\int_0^1\ln |f(e^{2\pi
ix})|dx-\epsilon\right).
\]
This implies that
\[
r(uf(v))\geq \lim_{n\rightarrow\infty}\max_{x\in
E}\left|f(\alpha^{n-1}e^{2\pi ix})f(\alpha^{n-2}e^{2\pi ix})\cdots
f(e^{2\pi ix})\right|^{1/n}\geq \exp\left(\int_0^1\ln |f(e^{2\pi
ix})|dx-\epsilon\right).
\]
Since $\epsilon>0$ is arbitrary, $r(uf(v))\geq \exp(\int_0^1\ln
|f(e^{2\pi ix})|dx)=\Delta(f(v))$.
\end{proof}

Recall that a continuous transformation $T: X\rightarrow X$ of a compact metrisable space $X$ is called
\emph{uniquely ergodic} if there is only one $T$ invariant Borel probability measure $\mu$ on $X$. If $T$ is uniquely ergodic, then
$\frac{1}{n}\sum_{i=0}^{n-1}f(T^ix)$ converges uniformly to $\int_Xf(x)d\mu(x)$
for every $f\in C(X)$ (see Theorem 6.19 of~\cite{Wa}). It is well-known that the irrational rotation of the unit circle is uniquely
ergodic. If we apply the above fact to $\ln |f(z)|$, then we easily see the following lemma.

\begin{Lemma}\label{L:geq}
If both $f(z), f(z)^{-1}\in C(S^1)$ and $\epsilon>0$, then
there exists an $N\in \mathbb{N}$ such that for all $n\geq N$ and
all $x\in [0,1]$
\[
\left(\left|f(\alpha^{n-1}e^{2\pi ix})f(\alpha^{n-2}e^{2\pi
ix})\cdots f(e^{2\pi ix})\right|\right)^{1/n}\leq
\exp\left(\int_0^1\ln |f(e^{2\pi ix})|dx+\epsilon\right).
\]
\end{Lemma}

\begin{Theorem}\label{T:spectral radius}
If $f(z)\in C(S^1)$, then $r(uf(v))=\Delta(f(v))$.
\end{Theorem}
\begin{proof}
We may assume that $|f(z)|\leq 1$ for all $z\in S^1$. For $k\in
\mathbb{N}$, define $f_k(z)=\max\{|f(z)|,\frac{1}{k}\}$. Then
$f_k(z), f_k(z)^{-1}\in C(S^1)$. Let $\epsilon>0$. By
Lemma~\ref{L:geq}, there exists an $N\in \mathbb{N}$ such that
for $n\geq N$ and all $x\in [0,1]$,
\[
\left|f(\alpha^{n-1}e^{2\pi ix})f(\alpha^{n-2}e^{2\pi ix})\cdots
f(e^{2\pi ix})\right|^{1/n}\leq \left|f_k(\alpha^{n-1}e^{2\pi
ix})f_k(\alpha^{n-2}e^{2\pi ix})\cdots f_k(e^{2\pi
ix})\right|^{1/n}\]\[\leq \exp\left(\int_0^1\ln f_k(e^{2\pi
ix})dx+\epsilon\right).
\]
Let $n\rightarrow\infty$, then \[r(uf(v))\leq
\exp\left(\int_0^1\ln f_k(e^{2\pi ix})dx+\epsilon\right).\]
Since $\epsilon>0$ is arbitrary, \[r(uf(v))\leq
\exp\left(\int_0^1\ln f_k(e^{2\pi ix})dx\right),\quad \forall
k\in\mathbb{N}.\] Note that
\[1=\frac{1}{f_1(e^{2\pi ix})}\leq \frac{1}{f_2(e^{2\pi
ix})}\leq \cdots\leq \frac{1}{f_n(e^{2\pi ix})}\leq \cdots,\] and
\[
\lim_{k\rightarrow\infty} \frac{1}{f_k(e^{2\pi
ix})}=\frac{1}{|f(e^{2\pi ix})|},\quad\forall x\in [0,1].
\]
So
\[
0\leq -\ln{f_1(e^{2\pi ix})}\leq -\ln{f_2(e^{2\pi ix})}\leq
\cdots\leq -\ln{f_n(e^{2\pi ix})}\leq \cdots,
\]
and \[ \lim_{k\rightarrow\infty} -\ln{f_k(e^{2\pi
ix})}=-\ln{|f(e^{2\pi ix})|},\quad \forall x\in[0,1].\] The
monotone convergence theorem implies that
\[
\lim_{k\rightarrow\infty} \int_0^1-\ln{f_k(e^{2\pi
ix})}dx=-\int_0^1\ln{|f(e^{2\pi ix})|}dx
\]
and therefore,
\[
r(uf(v))\leq \lim_{k\rightarrow\infty}\exp\left(\int_0^1\ln
f_k(e^{2\pi ix})dx\right)= \exp\left(\int_0^1\ln |f(e^{2\pi
ix})|dx\right).
\]
Now if $\int_0^1\ln |f(e^{2\pi ix})|dx=-\infty$, then
$r(uf(v))\leq \exp\left(\int_0^1\ln |f(e^{2\pi ix})|dx\right)=0$
and hence $r(uf(v))=\exp\left(\int_0^1\ln |f(e^{2\pi
ix})|dx\right)=0$. If $\int_0^1\ln |f(e^{2\pi ix})|dx>-\infty$,
then $\int_0^1\left|\left(\ln |f(e^{2\pi
ix})|\right)\right|dx<\infty$. By Lemma~\ref{L:leq}, $r(uf(v))\geq
\exp\left(\int_0^1\ln |f(e^{2\pi ix})|dx\right)$ and hence
$r(uf(v))=\exp\left(\int_0^1\ln |f(e^{2\pi
ix})|dx\right)=\Delta(f(v))$.
\end{proof}

\section{Spectrum of $uf(v)$}

\begin{Lemma}\label{L:less than 1}
Let $f_n(z)\in L^2(S^1,m)$ for $n\in \mathbb{Z}$ and
assume $T=\sum_{n=-\infty}^\infty u^nf_n(v)\in R$. Then for each $n\in \mathbb{Z}$,
\begin{equation}\label{E:geq 1}
\|f_n(v)\|\leq \|T\|.
\end{equation}
\end{Lemma}
\begin{proof}
Let $N$ be the von Neumann algebra generated by $v$. For any $n\in \mathbb{Z}$, we have
\[
\|f_n(v)\|=\|E_N(u^{-n}T)\|\leq \|T\|.
\]
\end{proof}

\begin{Lemma}\label{L:connected}
Let $f(z)\in L^\infty(S^1,m)$ and let $x=uf(v)$. Then the spectrum $\sigma(x)$ of $x$ is connected.
\end{Lemma}
\begin{proof}
Suppose the spectrum $\sigma(x)$ of $x$ is not connected. Then the Riesz spectral decomposition theorem gives a nontrivial idempotent $p$ in the Banach algebra generated by $x$.  Let $\tau$ be the faithful trace on $R$. We may assume that
$0<\tau(p)\leq 1/2$. Let $\epsilon>0$ be sufficiently small and let $a=\lambda+\sum_{n=1}^N\lambda_nx^n$ such that $\|p-a\|<\epsilon/(2\|p\|^2+2)$. Since $|\tau(a)-\tau(p)|\leq \|p-a\|<\epsilon/(2\|p\|^2+2)$, we may assume that $0\leq \lambda<3/4$. Then
\[
\|a^2-a\|\leq \|a^2-pa\|+\|pa-p^2\|+\|p-a\|\leq ((\|p\|+\epsilon)+\|p\|+1)\|p-a\|<\epsilon.
\]
This implies that
\[
\left\|\left(\sum_{n=1}^N\lambda_nx^n+\lambda\right)
\left(\sum_{n=1}^N\lambda_nx^n+\lambda\right)-
\left(\sum_{n=1}^N\lambda_nx^n+\lambda\right)\right\|<\epsilon.
\]
By Lemma~\ref{L:less than 1}, we have
$|\lambda^2-\lambda|<\epsilon.$
Since $0<\lambda<3/4$, $\lambda<\lambda^2+\epsilon<\frac{3}{4}\lambda+\epsilon$. Therefore $0\leq \lambda<4\epsilon$.
 Thus $\tau(p)\leq \tau(a)+\epsilon/(2\|p\|^2+2)<5\epsilon$. Since $\epsilon>0$ is arbitrary, $\tau(p)=0$. So $p=0$. This is a contradiction.
\end{proof}

\begin{Theorem}\label{T:spectrum}
Let $f(z)\in C(S^1)$ and let $x=uf(v)$. Then the spectrum
of $x$ is given as follows:
\begin{enumerate}
\item If $f(v)$ is invertible, then
$\sigma(uf(v))=\Delta(f(v))S^1$. \item If $f(v)$ is not
invertible, then $\sigma(uf(v))=\overline{{\mathbb
B}(0,\Delta(f(v)))}$.
\end{enumerate}
Here $\Delta(f(v))$ is the Fuglede-Kadison determinant of $f(v)$.
\end{Theorem}
\begin{proof}
Suppose $f(v)$ is not invertible, then $0\in \sigma(uf(v))$. By
Theorem~\ref{T:spectral radius}, we have $r(uf(v))=\Delta(f(v))$.
Clearly $\sigma(uf(v))$ is rotation symmetric. By
Lemma~\ref{L:connected}, $\sigma(uf(v))=\overline{{\mathbb
B}(0,\Delta(f(v))}$. Suppose $f(v)$ is invertible. By
Theorem~\ref{T:spectral radius}, $r(uf(v))=\Delta(f(v))$. Note
that $(uf(v))^{-1}=f(v)^{-1}u^*=u^*f(e^{-2\pi i\theta}v)^{-1}$. So
\[
r((uf(v))^{-1})=\Delta(f(e^{-2\pi
i\theta}v)^{-1})=\exp\left(\int_0^1\ln |f(e^{-2\pi i\theta}e^{2\pi
ix})|^{-1}dx \right)=\exp\left(\int_0^1\ln |f(e^{2\pi
ix})|^{-1}dx\right)\]\[=\Delta(f(v)^{-1})=\left(\Delta(f(v))\right)^{-1}.
\]
So $\sigma(uf(v))$ is contained in $\Delta(f(v))S^1$. Since
$\sigma(uf(v))$ is rotation symmetric,
$\sigma(uf(v))=\Delta(f(v))S^1$.
\end{proof}

\begin{Remark}
\emph{A natural question is that if the above theorem can be generalized to $f\in L^\infty(S^1)$. It can be shown that the above
theorem can be generalized to a larger class of functions which are essentially an upper semi-continuous functions. However the formula in
the above theorem does not hold for some $f\in L^\infty(S^1,m)$. Indeed one can construct a proper open subset $E$ of $S^1$ such that $r(u\chi_E(v))=1$ but
$\Delta(\chi_E)=0$.
}
\end{Remark}

\section{Von Neumann algebras generated by $uf(v)$}

\begin{Lemma}
Let $f\in L^\infty(S^1)$. If $f(v)$ is  not a scalar operator, then the von Neumann subalgebra generated by $u$ and $f(v)$ is an irreducible subfactor of $R$.
\end{Lemma}
\begin{proof}
Let $N$ be the von Neumann algebra generated by $u$ and $f(v)$. Suppose $x\in R$ commutes with $N$. Then $xu=ux$. Since the von Neumann subalgebra generated by $u$ is a maximal abelian von Neumann subalgebra of $R$, $x$ is in the von Neumann algebra generated by $u$. So $x=\sum_{n\in \mathbb{Z}}\alpha_n u^n$, where $\sum_{n\in \mathbb{Z}}|\alpha_n|^2<\infty$. Now $xf(v)=f(v)x$ implies that
\[
\sum_{n\in \mathbb{Z}}\alpha_n u^nf(v)=\sum_{n\in \mathbb{Z}}\alpha_nu^nf(e^{2\pi in\theta}v).
\]
So $\alpha_n u^nf(v)=\alpha_nu^nf(e^{2\pi in\theta}v)$ for all $n\in \mathbb{Z}$. If for some $n\neq 0$ we have $\alpha_n\neq 0$, then $f(v)=f(e^{2\pi in\theta}v)$. This implies that $f(v)$ is a scalar operator, which contradicts to the assumption. Therefore, $x$ is a scalar operator and $N$ is an irreducible subfactor of $R$.
\end{proof}

\begin{Theorem}\label{T:Liu}
Let $f\in L^\infty(S^1)$ such that $f(v)$ is not a scalar operator. Let $N$ be the irreducible subfactor of $R$ generated by $u$ and $f(v)$. Then $[R:N]$ is a finite integer. Furthermore, the following conditions are equivalent:
\begin{enumerate}
\item $f(z)$ is a periodic function with minimal period $e^{2\pi i/n}$;
\item Suppose $f(z)=\sum_{k\in \mathbb{Z}} \alpha_k z^k$ is the Fourier series of $f(z)$. Then $n={\rm gcd}\{k:\,\alpha_k\neq 0\}$;
\item $N=W^*(u,v^n)$.
\end{enumerate}
\end{Theorem}
\begin{proof}
$(1)\Leftrightarrow (2)$. Suppose $f(z)$ is a periodic function with minimal period $e^{2\pi i/n}$ and $m={\rm gcd}\{k:\,\alpha_k\neq 0\}$.
 Then $f(z)=f(e^{2\pi i/m}z)$ for almost all $z\in S^1$. So $f(z)$ is a periodic function period $e^{2\pi i/m}$. Since $e^{2\pi i/n}$ is a minimal period of $f(z)$, $n=mj$ for some positive integer $j$. Suppose $j\geq 2$. Then there exists $k_0$ such that $\alpha_{k_0}\neq 0$ and $n$ is not a factor of $k_0$. Since $f(z)=f(e^{2\pi i/n}z)$, by comparing the coefficients of both sides, we have
\[
\alpha_{k_0}z^{k_0}=\alpha_{k_0}e^{2\pi ik_0/n}z^{k_0}.
\]
This is a contradiction. Thus $n=m$.

$(2)\Rightarrow (3)$. Note that
\[
(u^*)^kf(v)u^k=f(e^{2\pi i k\theta}v).
\]
So $f(e^{2\pi i k\theta}v)\in N$ for all $k\in \mathbb{Z}$. Since $\{e^{2\pi i k\theta}: k\in \mathbb{Z}\}$ is dense in the unit circle $S$, $f(e^{2\pi it}v)\in N$ for all $t\in [0,1]$. For $k\in \mathbb{Z}$ and $z\in S^1$,
\[
g_k(z)=\int_0^1 e^{2\pi ik t}f(ze^{-2\pi it})dt=\alpha_k z^k.
\]
Thus if $\alpha_k\neq 0$, then $\alpha_k^{-1}g_k(v)=v^k\in N$. Since $n={\rm gcd}\{k:\,\alpha_k\neq 0\}$, $v^n\in N$. This proves that $W^*(u,v^n)\subseteq N$. Clearly, $N\subseteq W^*(u,v^n)$. So $N=W^*(u,v^n)$.

$(3)\Rightarrow (2)$. Suppose $N=W^*(u,v^n)$ and $m={\rm gcd}\{k:\,\alpha_k\neq 0\}$. By $(2)\Rightarrow (3)$, $W^*(u,v^n)=N=W^*(u,v^m)$. Hence $m=n$.
\end{proof}

\begin{Corollary}\label{C:Liu}
Suppose the zero set of $f(z)\in L^\infty(S^1)$ has Lebesgue measure zero and $|f|(v)$ is not a scalar operator. Then $W^*(uf(v))$ is an irreducible subfactor of $R$ with index $n$ for some positive integer $n$.
\end{Corollary}
\begin{proof}
Let $N$ be the von Neumann subalgebra generated by $uf(v)$ and let $f(v)=w|f|(v)$ be the polar decomposition of $f(v)$. Then $R=\{uw,v\}''$ and $vuw=e^{2\pi i\theta}uw v$. Note that $|f|^2(v)=(uf(v))^*(uf(v))\in N$. So $|f|(v)\in N$ and $|f|^{-1}(v)$ is an (unbounded) operator affiliated with $N$. Thus $uw=uw|f|(v)|f|^{-1}(v)$ is a bounded operator affiliated with $N$. Therefore, $uw\in N$. By Theorem~\ref{T:Liu}, $N=W^*(uw,|f|(v))$ is an irreducible subfactor of $R$ with index $n$ for some positive integer $n$.
\end{proof}

The above corollary shows that under a very general assumption $W^*(uf(v))$ is an irreducible subfactor of $R$. Recall that
an operator $x\in R$ is called a \emph{strongly irreducible operator relative to R} if there does not exist nontrivial idempotents $p$ in $\{x\}'\cap R$~\cite{FJX}. So an operator $x\in R$ is strongly irreducible relative to $R$ if and only if for every invertible operator $z\in R$, $W^*(zxz^{-1})$ is an irreducible subfactor of $R$.

\begin{Theorem}\label{T:strongly irreducible operators}
Suppose $f(z)\in C(S^1)$ such that the zero set of $f(z)$ has Lebesgue measure zero and is nonempty. Then $uf(v)$ is strongly irreducible relative to $R$.
\end{Theorem}
\begin{proof}
Let $x=uf(v)$ and $y=\sum_{n=-\infty}^\infty u^nf_n(v)\in \{x\}'\cap R$.  We have $f_0(e^{2\pi i\theta}v)f(v)=f(v)f_0(v)$ by comparing the coefficient of $u$. By the assumption of the theorem, $f_0(e^{2\pi i\theta}v)=f_0(v)$. By the ergodicity of irrational rotation, $f_0(v)=\lambda_0$ for some complex number $\lambda_0$. We have $f(e^{2(n-1)\pi i\theta}v)f_n(v)=f_n(e^{2\pi i\theta}v)f(v)$ by comparing the coefficient of $u^{n+1}$ for $n\geq 1$. Thus,
\[
\frac{f_n(e^{2\pi i\theta}v)}{f(e^{2\pi in\theta}v)f(e^{2\pi i(n-1)\theta}v)\cdots f(e^{2\pi in\theta}v)}=\frac{f_n(v)}{f(e^{2\pi i(n-1)\theta}v)f(e^{2\pi i(n-2)\theta}v)\cdots f(v)}.
\]
By the ergodicity of irrational rotation, $f_n(v)=\lambda_n f(e^{2\pi i(n-1)\theta}v)f(e^{2\pi i(n-2)\theta}v)\cdots f(v)$ for a complex number $\lambda_n$. So $u^nf_n(v)=\lambda_n x^n$.
We have $f_{-n}(e^{2\pi i\theta}v)f(v)=f(e^{-2\pi in\theta}v)f_{-n}(v)$ by comparing the coefficient of $u^{-(n-1)}$ for $n\geq 1$. Similarly, $f_{-n}(v)=\lambda_{-n}\left(f(e^{-2\pi in\theta}v)\cdots f(e^{-2\pi i\theta}v)\right)^{-1}$. By the assumption of the theorem, $\lambda_{-n}=0$. Thus $y=\sum_{n=0}^\infty \lambda_n x^n$. If $y^2=y$, then clearly $y=0$ or $y=1$. So $x$ is strongly irreducible relative to $R$.
\end{proof}

Now we have the following corollary, which is firstly proved in~\cite{FJX} by a different method.
\begin{Corollary}
$u+v$ is strongly irreducible relative to $R$.
\end{Corollary}

\section{Invariant subspaces of $uf(v)$ relative to $R$}

The following theorem is  Theorem 2.2 of~\cite{HS2}.
\begin{Theorem}\label{T:HS}
Let $T\in M$, and for $n\in \mathbb{N}$, let $\mu_n\in
\text{Prob}([0,\infty))$ denote the distribution of $(T^n)^*T^n$
w.r.t $\tau$, and let $\nu_n$ denote the push-forward measure of
$\mu_n$ under the map $t\rightarrow t^{\frac{1}{n}}$. Moreover,
let $\nu$ denote the push-forward measure of $\mu_T$ under the map
$z\rightarrow |z|^2$, i.e., $\nu$ is determined by
\[
\nu([0,t^2])=\mu_T(\overline{{\mathbb B}(0,t)}),\quad t>0.
\]
Then $\nu_n\rightarrow \nu$ weakly in $\text{Prob}([0,\infty))$.
\end{Theorem}

\begin{Lemma}\label{L:Birkhoff}
If $f(z)\in L^\infty(S^1)$, then for almost all $z\in S^1$,
\[
\lim_{n\rightarrow\infty}\left|\prod_{k=0}^{n-1}|f|(\alpha^kz)\right|^{\frac{1}{n}}=\Delta(f(v)).
\]
\end{Lemma}
\begin{proof}
 If $\int_0^1\ln |f(e^{2\pi ix})|dx>-\infty$,
then $\int_0^1\left|\left(\ln |f(e^{2\pi
ix})|\right)\right|dx<\infty$. By Birkhoff's ergodic theorem, the lemma holds. If $\int_0^1\ln |f(e^{2\pi ix})|dx=-\infty$,
then $\Delta(f(v))=0$. We may assume that $|f(z)|\leq 1$ for all $z\in S^1$. For $m\in
\mathbb{N}$, define $f_m(z)=\max\{|f(z)|,\frac{1}{m}\}$. For each $m$ and every $z$,
\[
\limsup_{n\rightarrow\infty}\left|\prod_{k=0}^{n-1}|f|(\alpha^kz)\right|^{\frac{1}{n}}\leq \lim_{n\rightarrow\infty}\left|\prod_{k=0}^{n-1}f_m(\alpha^kz)\right|^{\frac{1}{n}},
\]
hence for almost all $z\in S^1$,
\[
\limsup_{n\rightarrow\infty}\left|\prod_{k=0}^{n-1}|f|(\alpha^kz)\right|^{\frac{1}{n}}\leq \Delta(f_m(v)).
\]
By the proof of Theorem~\ref{T:spectral radius}, $\lim_{m\rightarrow\infty}\Delta(f_m(v))=\Delta(f(v))=0$. So for almost all $z\in S^1$,
\[
\limsup_{n\rightarrow\infty}\left|\prod_{k=0}^{n-1}|f|(\alpha^kz)\right|^{\frac{1}{n}}\leq 0.
\]
This implies the lemma.
\end{proof}

\begin{Theorem}\label{T:Brown measure}
If $f(z)\in L^\infty(S^1)$, then the Brown measure of $uf(v)$ is
the Haar measure on the circle $\Delta(f(v))S^1$.
\end{Theorem}
\begin{proof}
 Let $T=uf(v)$, and let $\nu$ and $\nu_n$ be the
measures defined as in Theorem~\ref{T:HS}. Then $\nu_n$ converges weakly to $\nu$. On the other hand,
$((T^n)^*T^n)^{\frac{1}{n}}=|f(v)\cdots
f(\alpha^{n-1}v)|^{\frac{2}{n}}$, where $\alpha=e^{2\pi i\theta}$.
So we can view $((T^n)^*T^n)^{\frac{1}{n}}$ as the multiplication
operator on $L^2[0,1]$ corresponding to the function$
\left|\prod_{k=0}^{n-1}\left(|f|^2(\alpha^kz)\right)\right|^{\frac{1}{n}}.$
By  Lemma~\ref{L:Birkhoff},  for almost all $z\in S^1$,
\[
\lim_{n\rightarrow\infty}\left|\prod_{k=0}^{n-1}\left(|f|^2(\alpha^kz)\right)\right|^{\frac{1}{n}}=\Delta(f(v))^2.
\]
Thus $\nu_n$ converges weakly to the Dirac measure $\delta_{\Delta(f(v))^2}$ in ${\rm Prob}([0,\infty))$.
Therefore, $\nu$ is the Dirac
measure $\delta_{\Delta(f(v))^2}$ and the support of $\mu_T$ is
contained in $\Delta(f(v))S^1$. Since $\mu_T$ is rotation
invariant, $\mu_T$ is the Haar measure on $\Delta(f(v))S^1$.
\end{proof}

In~\cite{Bo}, the spectrum and Brown measure of $u+\lambda v$ are calculated.
As an application of Theorem~\ref{T:spectrum} and Theorem~\ref{T:Brown measure}, we give another
 method to calculate the spectrum and Brown spectrum of $u+\lambda v$.
 Note that our method is very different from the method used in~\cite{Bo}, which uses analytical function theory
to calculate the spectrum  and Brown spectrum of $u+\lambda v$.

\begin{Corollary}
The spectrum of $u+\lambda v$ is
\[
\sigma(u+\lambda v)=
\begin{cases}
S^1&|\lambda<1|;\\
\overline{\mathbb{B}(0,1)}& |\lambda|=1;\\
\lambda S^1& |\lambda|>1,
\end{cases}
\]
and the Brown spectrum of $u+\lambda v$ is the Haar measure on $\lambda S^1$.
\end{Corollary}

Combining with the main result of~\cite{HS2} and Theorem~\ref{T:Brown measure}, we have the following
\begin{Corollary}
If $\Delta(f(v))>0$, then $uf(v)$ has a continuous family of hyperinvariant subspace affiliated with $R$. In particular, if $f$ is a
polynomial then $\Delta(f(v))>0$.
\end{Corollary}
\begin{proof}
Suppose $f(z)$ is a polynomial. Then $f(z)=\alpha (z-z_1)\cdots(z-z_n)$. So $\Delta(f(v))=|\alpha| \Delta(v-z_1)\cdots \Delta(v-z_n)$. If $|z_i|\neq 1$, then $v-z_i$ is an invertible operator. Therefore, $\Delta(v-z_i)>0$. If $|z_i|=1$, then $\Delta(v-z_i)=\Delta(v-1)=1$. Thus $\Delta(f(v))>0$.
\end{proof}

On the other hand, there are $f\in C(S^1)$ such that $\Delta(f(v))=0$. For example, for $p\geq 1$, let
\[
g(x)=\begin{cases}
0& x=0;\\
\exp\left(-\frac{1}{x^p}\right)&0<x\leq \frac{1}{2};\\
\exp\left(-\frac{1}{(1-x)^p}\right)&\frac{1}{2}\leq x<1;\\
0& x=1.
\end{cases}
\]
Then $g(x)$ is a continuous function on $[0,1]$ and $g(x)=g(1-x)$ for $x\in[0,1]$. Therefore, there exists a continuous function $f(z)$ on $S^1$ with
a single zero point such that $f\left(e^{2\pi ix}\right)=g(x)$. Now
\[
\Delta(f(v))=\exp\left(\int_0^1 \ln f\left(e^{2\pi ix}\right)dx\right)=0.
\]
In this case the Brown measure of $uf(v)$ is the Dirac measure. So the main result of~\cite{HS2} does not apply to this case. In the following we will show that indeed the well-known methods can not determine whether $uf(v)$ has a nontrivial invariant subspace affiliated with $R$ in this case.

Recall that if $M$ is a von Neumann algebra acting on a Hilbert space $\mathcal{H}$ and $T\in M$. A Haagerup's invariant subspace of $T$ is defined by
\cite{DH2,Ga}
\[
\mathcal{E}_r(T):=\left\{\xi\in\mathcal{H}: \limsup_n \gamma_n\|T^n(\xi)\|^{1/n}\leq r \right\}\quad\text{and}\quad \mathcal{H}_r(T)=\overline{\mathcal{E}_r(T)}.
\]
This subspace is closed, $T$-invariant, affiliated to $M$ and moreover, hyperinvariant. However, we will prove that for any sequence $\{\gamma_n\}_n$ and for any $r>0$ this subspace is trivial for $uf(v)$.

\begin{Proposition}
Let $r>0$ and $\{\gamma_n\}_n$ be a sequence of positive numbers. The subspace $\mathcal{H}_r(uf(v))$ defined as above is either $\mathcal{H}$ or $\{0\}$ if the zero points of $f(z)$ has Lebesgue measure zero and $|f|(z)$ is not a scalar operator.
\end{Proposition}
\begin{proof}
First we show that $\mathcal{H}_r(uf(v))$ is also an invariant subspace of $(uf(v))^*$. Suppose
\[\limsup_n\gamma_n\|(uf(v))^n\xi\|^{1/n}\leq r.\]
Let $f(v)=w|f|(v)$. Then $R=\{u,v\}''=\{uw,v\}''$ and $v(uw)=e^{2\pi i\theta}(uw)v$. Replacing $u$ by $uw$ and $f(v)$ by $|f|(v)$, we may assume that $f(v)$ is a positive operator. So $(uf(v))^*=f(v)u^*$. Let $\alpha=e^{2\pi i\theta}$. Then
\[
\|(uf(v))^n(uf(v))^*\xi\|=\|u^{n-1}f(\alpha^{n-2}v)\cdots f(v)f^2(\alpha^{-2}v)\xi\|=\|f(\alpha^{n-2}v)\cdots f(v)f^2(\alpha^{-2}v)\xi\|
\]
\[
\leq \|f^2(\alpha^{-2}v)\| \|f(\alpha^{n-2}v)\cdots f(v)\xi\|.
\]
So
\[
\|(uf(v))^n(uf(v))^*\xi\|^{1/n}\leq \|f^2(\alpha^{-2}v)\|^{1/n} \|f(\alpha^{n-2}v)\cdots f(v)\xi\|^{1/n}.
\]
Note that
\[
\|(uf(v))^n\xi\|=\|u^nf(\alpha^{n-1}v)\cdots f(v)\xi\|=\|f(\alpha^{n-1}v)\cdots f(v)\xi\|.
\]
Therefore,
\[
\limsup_n \gamma_n\|(uf(v))^n(uf(v))^*\xi\|^{1/n}\leq r.
\]
This proves that $\mathcal{H}_r(uf(v))$ is also an invariant subspace of $(uf(v))^*$. So the projection $P_{\mathcal{H}_r}$ is in the commutant algebra of $W^*(uf(v))$. By Corollary~\ref{C:Liu}, $\mathcal{H}_r(uf(v))$ is either $\mathcal{H}$ or $\{0\}$.
\end{proof}

In~\cite{Ga}, Gabriel introduced a class of quasinilpotent operators in the hyperfinite type ${\rm II}_1$ factor. He showed for such operators one can find a nilpotent operator $S$ in the commutant algebra of the quasinilpotent operator. Thus the range projection of $S$ is a nontrivial invariant subspace of such operator. The following result tells us this idea does not apply to our case.

\begin{Proposition}
Suppose $f(z)\in C(S^1)$ such that the zero set of $f(z)$ has Lebesgue measure zero and is nonempty. If $S$ is a nilpotent operator in the commutant algebra of $uf(v)$, then $S=0$.
\end{Proposition}
\begin{proof}
By the proof of Theorem~\ref{T:strongly irreducible operators}, $S=\sum_{n=0}^\infty \alpha_n (uf(v))^n$. So if $S$ is a nilpotent operator then $S=0$.
\end{proof}

\noindent {\bf Question:}\, Suppose $\Delta(f(v))=0$, the zero points of $f(z)$ has Lebesgue measure zero and $|f|(z)$ is not a scalar operator. Does $uf(v)$ have a nontrivial invariant subspace affiliated with $R$?

\section{$C^*$-algebras generated by $uf(v)$ and $1$}
In this section, $f(z)\in C(S^1)$. Let  $C^*(uf(v),1)$ be the $C^*$-algebra generated by $uf(v)$ and $1$, and  let $A=C^*(uf(v),1)\cap C^*(v)$. Then $A$ is a unital $C^*$-subalgebra of $C^*(v)$ and  $C^*(uf(v),1)=C^*(uf(v),A)$. Recall that $\alpha=e^{2\pi i\theta}$. Note that $A$ is the $C^*$-algebra generated by $1$, $(uf(v))^*uf(v)$, $uf(v)(uf(v))^*$, $[(uf(v))^*]^2(uf(v))^2$, $(uf(v))^2[(uf(v))^*]^2$, $\cdots$. Simple calculations show that
\begin{equation}\label{E:1}
A=C^*(1, |f|(v), |f|(\alpha^{-1}v), |f|(v)|f|(\alpha v), |f|(\alpha^{-1}v)|f|(\alpha^{-2}v), |f|(v)|f|(\alpha v)|f|(\alpha^2 v), \cdots).
\end{equation}
 In the following we identify $C^*(v)$ with $C(S^1)$ by the Gelfand theorem and thus we view $A$ as a unital subalgebra of $C(S^1)$.

\begin{Theorem}
If $f(v)$ is an invertible operator, then $C^*(uf(v),1)\cong C^*(u,v^n)$ for some $n=0,1,2,\cdots$. Furthermore, if $|f|(v)$ is not a periodic function then $C^*(uf(v),1)=C^*(u,v)$.
\end{Theorem}
\begin{proof}
If $|f|(v)$ is a scalar operator, then $uf(v)$ is a Haar unitary operator. Therefore, $C^*(uf(v),1)\cong C^*(u)$. Assume that $|f|(v)$ is a non-scalar invertible operator. Then $f(v)=u_1|f|(v)$ for some unitary operator $u_1\in C^*(v)$. So $uu_1=uf(v)|f|(v)^{-1}\in C^*(uf(v),1)$. By~(\ref{E:1}),
\[A=C^*(uf(v),1)\cap C^*(v)=C^*\{|f|(\alpha^kv):\, k\in \mathbb{Z}\}.\]
 If $A$ separates points of $S^1$, then the Stone-Weierstrass theorem implies that $A=C^*(v)$. Thus $C^*(uf(v),1)=C^*(uf(v), v)=C^*(u,v)$.  Now suppose $A$ does not separate points of $S^1$. Then there exists $z_1\neq z_2$, $z_1,z_2\in S^1$, such that $|f|(\alpha^kz_1)=|f|(\alpha^kz_2)$ for all $k\in \mathbb{Z}$. Since $\{\alpha^k:\,k\in \mathbb{Z}\}$ is dense in $S^1$, we have $f(zz_1)=f(zz_2)$. Let $z_0=z_2z_1^{-1}$ and replace $z$ by $zz_1^{-1}$. Then $|f|(z)=|f|(z_0z)$ for all $z\in S^1$. Suppose $|f|(z)=\sum_{k\in\mathbb{Z}}\alpha_kz^k$ is the Fourier series of $|f|(z)$. Then $|f|(z_0z)=\sum_{k\in\mathbb{Z}}\alpha_kz_0^kz^k=\sum_{k\in\mathbb{Z}}\alpha_kz^k$. If $\alpha_n\neq 0$, then $z_0^n=1$. Let $n={\rm gcd}\{k:\,\alpha_k\neq 0\}$. Then by the proof of Theorem~\ref{T:Liu}, $|f|(z)$ is a periodic function with a minimal period $e^{2\pi i/n}$. Since $\{\alpha^k:\, k\in\mathbb{Z}\}$ is dense in the unit circle $S$, $|f|(e^{2\pi it}v)\in A$ for all $t\in [0,1]$. For $k\in \mathbb{Z}$ and $z\in S^1$,
\[
g_k(z)=\int_0^1 e^{2\pi ikt}|f|(e^{-2\pi it}z)dt=\alpha_kz^k.
\]
Thus if $\alpha_k\neq 0$, then $\alpha_k^{-1}g_k(v)=v^k\in A$. Since $n={\rm gcd}\{k:\,\alpha_k\neq 0\}$, $v^n\in A$.  Conversely, since
$|f|(\alpha^k z)$ has period $e^{2\pi i/n}$, $|f|(\alpha^k v)\in C^*(v^n)$. Thus $A=C^*(v^n)$. Therefore,
\[
C^*(uf(v),1)=C^*(uf(v),A)=C^*(uu_1|f|(v), v^n)=C^*(uu_1,v^n)\cong C^*(u, v^n).
\]
\end{proof}

Let $Y$ be the zero points of $f(z)$. Then $Y$ is also the zero points of $|f|(z)$. In the following we assume that $Y\neq \emptyset$. Define $\phi(z)=\alpha z=e^{2\pi i\theta}z$. For $\xi\in S^1$ denote by
$$
Orb(\xi)=\{\phi^n(\xi): n\in {\mathbb Z}\}
$$
the orbit of $\xi$ under the rotation $\phi.$
By Proposition 2.5 of~\cite{FJX}, the following conditions are equivalent:

\begin{enumerate}
\item[(1).] $\phi^n(Y)\cap Y=\emptyset$ for any integer $n\not=0;$

\item[(2).] For each $\xi\in S^1,$ $Orb(\xi)\cap Y$ contains at most one point;

\item[(3).] $Y_1\cap Y_2=\emptyset$, where $Y_1=\cup_{n\ge 0}\phi^n(Y)$ and
    $Y_2=\cup_{k\ge 1} \phi^{-k}(Y).$

\end{enumerate}
By the proof of Corollary 4.6 of~\cite{FJX}, if $F$ is a Lebesgue measurable subset of $S^1$  satisfying the above conditions, then $m(F)=0$. Recall that $A=C^*(uf(v),1)\cap C^*(v)$.

\begin{Lemma}\label{L: structrue of A}
Let $Y$ be the zero points of $f(z)$. If $Y$ satisfies one of the above conditions (1)-(3), then $A=C^*(v^n)$ for some natural number $n\geq 1$. Furthermore, $A=C^*(v)$ if and only if $|f|(z)$ is  not a periodic function.
\end{Lemma}
\begin{proof}
By the Stone-Weierstrass theorem, if $A$ separates points of $S^1$ then $A=C^*(v)$. Otherwise, there exists $z_1\neq z_2$, $z_1,z_2\in S^1$, such that $g(z_1)=g(z_2)$ for all $g\in A$.  Suppose $|f|(\alpha^k z_1)=|f|(\alpha^k z_2)\neq 0$ for all $k=0,1,2,\cdots$. Since $\{\alpha^k:\, k\in \mathbb{N}\}$ is dense in $S^1$, $|f|(zz_1)=|f|(zz_2)$ for all $z\in S^1$. Replacing $z$ by $zz_1^{-1}$, we have $|f|(z)=|f|(zz_0)$, where $z_0=z_2z_1^{-1}$. Thus $|f|(z)$ is a periodic function with period $z_0$. Suppose $|f|(\alpha^k z_1)=|f|(\alpha^k z_2)=0$ for some $k=0,1,2,\cdots$. Then claim $|f|(\alpha^{-k}z_1)=|f|(\alpha^{-k}z_2)\neq 0$ for all $k\in \mathbb{N}$. Otherwise $|f|(\alpha^{-k'}z_1)=|f|(\alpha^{-k'}z_2)=0$ for some $k'\in \mathbb{N}$. Now both $\alpha^k z_1$ and $\alpha^{-k'}z_1$ belong to $Y$. This contradicts to condition (2) above the lemma. Thus $|f|(\alpha^{-k}z_1)=|f|(\alpha^{-k}z_2)\neq 0$ for all $k\in \mathbb{N}$. A similar argument shows that $|f|(z)$ is a periodic function. Let $e^{\frac{2\pi i}{n}}$ be a minimal period of $f(z)$. Claim $A=C^*(v^n)$. Let $X$ be the quotient space $\{e^{2\pi it}: t\in[0,\frac{2\pi}{n}]\}/\{1, e^{\frac{2\pi i}{n}}\}$.  Then $A$ and $v^n$ can be viewed as continuous functions on $X$. Note that $v^n$ separates points of $X$. By the Stone-Weierstrass theorem, $A\subseteq C^*(v^n)$. Claim $A$ also separates points of $X$.  Otherwise, there exists $z_1\neq z_2$, $z_1,z_2\in X$, such that $g(z_1)=g(z_2)$ for all $g\in A$.
By a similar argument we have  $|f|(z)=|f|(zz_0)$, where $z_0=z_2z_1^{-1}$. So $e^{\frac{2\pi i}{n}}$ is not a minimal period of $f(z)$. This is a contradiction. Thus $C^*(v^n)\subseteq A$ and hence $A=C^*(v^n)$.
\end{proof}

\begin{Theorem}\label{T:structure of C*(uf(v))}
Let $Y$ be the zero points of $f(z)$. If $Y$ satisfies one of the conditions (1)-(3) above Lemma~\ref{L: structrue of A}, then $C^*(uf(v),1)$ is a generalized universal irrational rotation $C^*$-algebra. Furthermore, if $|f|(z)$ is not a periodic function, then $C^*(uf(v),1)\cong A_{\theta,|f|^2}$.
\end{Theorem}
\begin{proof}
By Lemma~\ref{L: structrue of A}, if $|f|(z)$ is not a periodic function then $A=C^*(v)$ and if $|f|(z)$ is a periodic function, then $A=C^*(v^n)$ for some $n\geq 2$. In the first case, let $f(v)=u_1|f|(v)$ be the polar decomposition of $f(v)$. Since $Y$ satisfies one of the conditions above Lemma~\ref{L: structrue of A}, $m(Y)=0$. Thus $u_1$ is a unitary operator in the von Neumann algebra generated by $v$. So \[C^*(uf(v),1)=C^*(uf(v),A)=C^*(uu_1|f(v)|, v)\cong A_{\theta,|f|^2}.\] In the second case, $f(v)=u_1|f|(v)$ and $|f|\in C^*(v^n)$. So there exists a positive continuous function $g(z)$ on the unit circle such that $f(v)=g(v^n)$. Therefore, \[C^*(uf(v),1)=C^*(uf(v),A)=C^*(uu_1|f(v)|, v^n) =C^*(uu_1|g(v^n)|, v^n)
=C^*(uu_1|g(w)|, w)\cong A_{n\theta,|g|^2}.\]
\end{proof}

\begin{Proposition}
Suppose $|f|(z)$ is not a periodic function and $Y$ is the zero points of $f(z)$. Then the following conditions are equivalent:
\begin{enumerate}
\item $C^*(uf(v),1)$ is a simple algebra;
\item $\phi^n(Y)\cap Y=\emptyset$ for all integers $n\not=0;$
\item For each $\xi\in S^1,$ $Orb(\xi)\cap Y$ contains at most one point;
\item $Y_1\cap Y_2=\emptyset$, where $Y_1=\cup_{n\ge 0}\phi^n(Y)$ and
    $Y_2=\cup_{k\ge 1} \phi^{-k}(Y).$
\end{enumerate}
\end{Proposition}
\begin{proof}
$2\Leftrightarrow 3\Leftrightarrow 4$ follows from Proposition 2.5 of~\cite{FJX}. $4\Rightarrow 1$ follows from Theorem~\ref{T:structure of C*(uf(v))}. We need to prove $1\Rightarrow 4$. Suppose $Y_1\cap Y_2\neq \emptyset$. Let $x=uf(v)$ and $\gamma(z)=|f|^2(z)$. Then there exists $\lambda\in S^1$, $m\geq 0$, $n\geq 1$ such that $\lambda$ is a zero point of $\gamma(e^{2\pi in\theta}z)$ and $\gamma(e^{-2\pi im\theta}z)$.
Consider the subset
\[
J=\{\varphi(v)|\varphi(v)\in A\,\text{and}\, \varphi(e^{2\pi i
n\theta}\lambda)=\cdots=\varphi(\lambda)=\cdots=\varphi(e^{-2\pi i
m\theta}\lambda)=0\}
\]
of $C^*(v)$. By the definition of $A$ (see (6.1)), $|f|(\alpha^{-m}v)\cdots |f|(\alpha^{-1}v)|f|(v)|f|(\alpha v)\cdots |f|(\alpha^n v)\in J$. So $J$ is a nonempty ideal of $A$. Claim that
$I=C^*(x,1)JC^*(x,1)$ is a
two-sided ideal of $C^*(x,1)$. Otherwise, there
exists $\varphi_i(v)\in J$,
\[
a_i=\sum_{n=1}^K (x^*)^ng^i_{-n}(v)+g^i(v)+\sum_{n=1}^Kg^i_{n}(v)x^n,
\]
and
\[
b_i=\sum_{n=1}^K (x^*)^nh^i_{-n}(v)+h^i(v)+\sum_{n=1}^Kh^i_{n}(v)x^n,
\]
for sufficiently large $K\in \mathbb{N}$ such that
\[
\left\|\sum_{i=1}^Na_i\varphi_i(v)b_i-1\right\|<1,
\]
where $g^i_n,g^i,h^i_n,h^i\in C(\mathbb{T})$. By Lemma~\ref{L:less than 1}
and simple computations, we have
\[
\|\sum_{i=1}^N g^i_{-K}(e^{2\pi i K\theta}v)\varphi_i(e^{2\pi i K\theta}v)h^i_{K}(e^{2\pi i
K\theta}v)\gamma(e^{2\pi i (K-1)\theta}v)\cdots
\gamma(v)+\]\[g^i_{-(K-1)}(e^{2\pi i (K-1)\theta}v)\varphi_i(e^{2\pi i
(K-1)\theta}v)h^i_{K-1}(e^{2\pi i (K-1)\theta}v)\gamma(e^{2\pi i
(K-2)\theta}v)\cdots \gamma(v) +\cdots+
\]
\[
g^i_{-1}(e^{2\pi i\theta}v)\varphi_i(e^{2\pi i\theta}v)h^i_{1}(e^{2\pi
i\theta}v)\gamma(v)+g^i(v)\varphi_i(v)h^i(v)+g^i_1(v)\varphi_i(e^{-2\pi
i\theta}v)h^i_{-1}(v)\gamma(e^{-2\pi i\theta}v)+\cdots+
\]
\[
g^i_{K-1}(v)\varphi_i(e^{-2\pi i(K-1)\theta}v)h^i_{-(K-1)}(v)\gamma(e^{-2\pi
(K-1)i\theta}v)\cdots \gamma(e^{-2\pi i\theta}v)+
\]
\[
g^i_{K}(v)\varphi_i(e^{-2\pi iK\theta}v)h^i_{-K}(v)\gamma(e^{-2\pi
Ki\theta}v)\cdots \gamma(e^{-2\pi i\theta}v)-1\|<1.
\]
Let
\[
\psi(z)=\sum_{i=1}^Ng^i_{-K}(e^{2\pi i K\theta}z)\varphi_i(e^{2\pi i
K\theta}z)h^i_{K}(e^{2\pi i K\theta}z)\gamma(e^{2\pi i
(K-1)\theta}z)\cdots \gamma(z)+\cdots+
\]
\[
g^i_{-1}(e^{2\pi i\theta}z)\varphi_i(e^{2\pi i\theta}z)h^i_{1}(e^{2\pi
i\theta}z)\gamma(z)+g^i(z)\varphi_i(z)h^i(z)+g^i_1(z)\varphi_i(e^{-2\pi
i\theta}z)h^i_{-1}(z)\gamma(e^{-2\pi i\theta}z)+\cdots+
\]
\[
g^i_{K}(z)\varphi_i(e^{-2\pi iK\theta}z)h^i_{-K}(z)\gamma(e^{-2\pi
Ki\theta}z)\cdots \gamma(e^{-2\pi i\theta}z).
\]

Since $\varphi_i(z)\in J$, $\varphi_i(e^{2\pi i
n\theta}\lambda)=\cdots=\varphi_i(\lambda)=\cdots=\varphi_i(e^{-2\pi i
m\theta}\lambda)=0$. Note that $\gamma(e^{2\pi
in\theta}\lambda)=\gamma(e^{-2\pi im\theta}\lambda)=0$. So
$\psi(\lambda)=0$. Hence $\|\psi(z)-1\|\geq 1$ and
$\|\psi(v)-1\|\geq 1$. By Lemma~\ref{L:less than 1},
\[
\left\|\sum_{i=1}^Na_i\varphi_i(v)b_i-1\right\|\geq \|\psi(v)-1\|\geq 1.
\]
 This is a contradiction.
\end{proof}

\begin{Corollary}
If $A_{\theta,\gamma}=C^*(u\gamma^{1/2}(v),v)$ is a simple generalized universal irrational rotation $C^*$-algebra, then $A_{\theta,\gamma}$ is
 generated by an element $uf(v)$ and the identity operator for some $f(z)\in C(S^1)$.
\end{Corollary}
\begin{proof}
Since $A_{\theta,\gamma}=C^*(u\gamma(v)^{1/2},v)$ is simple, by Corollary 6.5 of~\cite{FJX} the zero points of $\gamma(z)^{1/2}$ satisfies conditions (1)-(3) above Lemma~\ref{L: structrue of A}. If $\gamma(z)$ is not a periodic function, then $\gamma(z)^{1/2}$ is not a periodic function. By Theorem~\ref{T:structure of C*(uf(v))}, \[A_{\theta,\gamma}=C^*(u\gamma(v)^{1/2},v)\cong C^*(u\gamma(v)^{1/2}, 1).\] If $\gamma(z)$ is a periodic function, then either $\gamma(z)|2+z|$ or $\gamma(z)|3+z|$ is not a periodic function. Otherwise $\frac{|3+z|}{|2+z|}$ will be a periodic function. Assume that $\gamma(z)|2+z|$ is not a periodic function. Then
 \[A_{\theta,\gamma}=C^*(u\gamma(v)^{1/2},v)=C^*(u\gamma(v)^{1/2}|2+v|^{1/2},v)\cong C^*(u\gamma(v)^{1/2}|2+v|^{1/2}, 1).\]
\end{proof}

\begin{Corollary}
Suppose $f(z)$ has a single zero point. Then $A=C^*(v)$ and $C^*(uf(v),1)=C^*(uf(v),v)\cong A_{\theta,|f|^2}$ is a simple generalized universal irrational rotation $C^*$-algebra.
\end{Corollary}

\begin{Lemma}\label{L:two zero points}
Suppose $f(z)$ has two zero points. Then $A=C^*(v)$ or $A=C^*(v^2)$. Furthermore, $A=C^*(v)$ if and only if $|f|(z)$ is not a periodic function.
\end{Lemma}
\begin{proof}
By the Stone-Weierstrass theorem, if $A$ separates points of $S^1$ then $A=C^*(v)$. Otherwise, there exists $z_1\neq z_2$, $z_1,z_2\in S^1$, such that $g(z_1)=g(z_2)$ for all $g\in A$. Since $f(z)$ has two zero points, $|f|(z)$ has two zero points. Suppose $|f|(\alpha^k z_1)=|f|(\alpha^k z_2)\neq 0$ for all $k=0,1,2,\cdots$. Since $\{\alpha^k:\, k\in \mathbb{N}\}$ is dense in $S^1$, $|f|(zz_1)=|f|(zz_2)$ for all $z\in S^1$. Replacing $z$ by $zz_1^{-1}$, we have $|f|(z)=|f|(zz_0)$, where $z_0=z_2z_1^{-1}$. Thus $|f|(z)$ is a periodic function with period $z_0$. Since $|f|(z)$ has exactly two zero points, $|f|(z)$ is periodic function with minimal period $e^{\pi i}$. Thus $A=C^*(v^2)$. Suppose $|f|(\alpha^k z_1)=|f|(\alpha^k z_2)=0$ for some $k=0,1,2,\cdots$. Then claim $|f|(\alpha^{-k}z_1)=|f|(\alpha^{-k}z_2)\neq 0$ for all $k\in \mathbb{N}$. Otherwise $|f|(\alpha^{-k'}z_1)=|f|(\alpha^{-k'}z_2)=0$ for some $k'\in \mathbb{N}$. Since $|f|(z)$ has exactly two zero points, $\{\alpha^k z_1,\alpha^k z_2\}=\{\alpha^{-k'}z_1,\alpha^{-k'}z_2\}$. Since $\alpha=e^{2\pi i\theta}$ and $\theta$ is irrational, this is impossible. Thus $|f|(\alpha^{-k}z_1)=|f|(\alpha^{-k}z_2)\neq 0$ for all $k\in \mathbb{N}$. A similar argument shows that $|f|(z)$ is periodic function with minimal period $e^{\pi i}$. Thus $A=C^*(v^2)$.
\end{proof}

\begin{Proposition}
Suppose $f(z)$ has two zero points. Then $C^*(uf(v),1)$ is a generalized universal irrational rotation $C^*$-algebra. Furthermore, if $|f|(z)$ is not a periodic function, then $C^*(uf(v),1)\cong A_{\theta,|f|^2}$.
\end{Proposition}
\begin{proof}
By Lemma~\ref{L:two zero points}, if $|f|(z)$ is not a periodic function then $A=C^*(v)$ and if $|f|(z)$ is a periodic function, then $A=C^*(v^2)$. In the first case, let $f(v)=u_1|f|(v)$ be the polar decomposition of $f(v)$.  Then $u_1$ is a unitary operator in the von Neumann algebra generated by $v$. So \[C^*(uf(v),1)=C^*(uf(v),A)=C^*(uu_1|f(v)|, v)\cong A_{\theta,|f|^2}.\] In the second case, $f(v)=u_1|f|(v)$ and $|f|\in C^*(v^2)$. So there exists a positive continuous function $g(z)$ on the unit circle such that $f(v)=g(v^2)$. Therefore, \[C^*(uf(v),1)=C^*(uf(v),A)=C^*(uu_1|f(v)|, v^2)= C^*(uu_1|g(v^2)|, v^2)
=C^*(uu_1|g(w)|, w)\cong A_{2\theta,|g|^2}.\]
\end{proof}

\noindent Junsheng Fang

\noindent School of Mathematical Sciences,
Dalian University of Technology,

\noindent Dalian, 116024, P.R China,

\noindent {\em E-mail address: } [Junsheng Fang]\,\,
junshengfang\@@gmail.com

\vspace{.2in}

\noindent Zhangsheng Zhu

\noindent School of Mathematical Sciences,
Dalian University of Technology,

\noindent Dalian, 116024, P.R China,

\noindent {\em E-mail address: } [Zhangsheng Zhu]\,\,
zhuzhangsheng2010\@@163.com

\vspace{.2in}

\noindent Rui Shi

\noindent School of Mathematical Sciences, Dalian University of
Technology,

\noindent Dalian, 116024, P.R China,

\noindent {\em E-mail address: } [Rui Shi]\,\,
littlestoneshg\@@gmail.com


\begin{thebibliography}{99}

 \bibitem{Bo} Boca, F., Rotation $C^*$-algebras and almost Mathieu operators. Theta Series in Advanced Mathematics, 1. The Theta Foundation, Bucharest, 2001. xviii+172 pp.

 \bibitem{Br} Brown, L.G., Lidskii's theorem in the type ${\rm II}$ case, Geometric methods in operator algebras, H.Araki and E.Effros(Eds.) Pitman Res. Notes Math. Ser 123 (1986), 1-35.


 \bibitem{BL} Biane, P.; Lehner, F., Computation of some examples of Brown's spectral measure in free probability, \emph{Colloq. Math.}, 90 (2001), 181-211.

\bibitem{DH} Dykema, K,; Haagerup, U., ${\rm DT}$-operators and decomposability of Voiculescu's circular operators, \emph{Amer. J. Math.}, 126 (2004) 121-189.


\bibitem{DH2} Dykema, K,; Haagerup, U., Invariant subspaces of the quasinilpotent ${\rm DT}$-operator, \emph{J.Funct.Anal.}, 209 (2004) 332-366.




 \bibitem{FHM} Fang, J.; Hadwin, D.; Ma, X., On spectra and Brown's spectral measures of elements in free products of matrix algebras, \emph{Math. Scand.}, (103) 2008, 77-96.




 \bibitem{FJX} Fang, J.; Jiang, C.L.; Hua, X.L.; Xu, F., On the operator $u+\lambda v$ and $C^*$-subalgebras of the universal irrational rotation algebra, \emph{International Journal of Mathematics}, (24) 2013, no 2, 1350059.


 \bibitem{FK} Fuglede, B.; Kadison, R.,  Determinant theory in finite factors, \emph{Ann. of Math.},55 (3) 1952, 520-530.

 \bibitem{Ga} Gabriel, H.T., Some quasinilpotent generators of the hyperfinite ${\rm II}_1$ factor, \emph{J. Funct. Anal.} 254 (2008), 2969-2994.

 \bibitem{HL} Haagerup, U.; Larsen, F.; Brown's spectral distribution measure for $R$-diagonal elements in finite von Neumann algebras, \emph{J.Funct.Anal}. 176 (2000), 331-367.


 \bibitem{HS2} Haagerup, U.; Schultz, H.; Invariant subspaces for operators in a general ${\rm II}_1$ factor, \emph{Publ. Math. Inst. Hautes $\acute{\text{E}}$tudes Sci}. No. 109 (2009), 19-111.


\bibitem{La} Last, Y., Spectral theory of Sturm-Liouville operators on infinite intervals: a review of recent developments. Sturm-Liouville theory, 99-120, Birkh$\ddot{\text{a}}$user, Basel, 2005.

\bibitem{NS} Nica, A.; Speicher, R.,  R-diagonal pairs, a common approach to Haar unitaries and circular elements, Fields Institute Commun. 12 (1997) 149-188.

\bibitem{SS} Sniady, P.; Speicher, R., Continuous family of invariant subspaces for $R$-diagonal operators, \emph{Invent. Math.} 146 (2001), no. 2, 329-363.

\bibitem{Wa} Walters, P.; An introduction to Ergodic Theory, Graduate texts in mathematics, 79.


 \end{thebibliography}
\end{document}